\documentclass{article}
\usepackage[intlimits]{amsmath}
\usepackage{latexsym,amsfonts,amsthm,amssymb, epsfig}
\usepackage{color,comment}
\usepackage{tabls}	

\newtheorem{thm}{Theorem}

\newtheorem{prop}[thm]{Proposition}

\theoremstyle{remark}

\theoremstyle{definition}

\newcommand{\R}{\mathbb{R}}
\newcommand{\Rn}{{\mathbb{R}^n}}
\newcommand{\N}{\mathbb{N}}
\newcommand{\C}{\mathbb{C}}

\newcommand{\vol}{\operatornamewithlimits{vol}}

\newlength{\fixboxwidth}
\title{Volumes of unit balls of mixed sequence spaces}
\author{Henning Kempka\footnote{Fakult\"at f\"ur Mathematik, TU Chemnitz, Reichenhainer Str. 39, 09126 Chemnitz, Germany, {\tt henning.kempka@mathematik.tu-chemnitz.de};
this author was supported by the German science foundation (DFG) within the project KE 1847/1-1.},
Jan Vyb\'iral\footnote{(Corresponding author) Department of Mathematical Analysis, Charles University, Sokolovsk\'a 83, 186 00, Prague 8, Czech Republic, {\tt vybiral@karlin.mff.cuni.cz};
this author was supported by the ERC CZ grant LL1203 of the Czech Ministry of Education.}}

\begin{document}
\maketitle
\begin{abstract}
The volume of the unit ball of the Lebesgue sequence space $\ell_p^m$ is very well known since the times of Dirichlet.
We calculate the volume of the unit ball in the mixed norm $\ell^n_q(\ell_p^m)$, whose special cases are nowadays popular in machine learning under the name
of group lasso. We consider the real as well as the complex case.
The result is given by a closed formula involving the gamma function, only slightly more complicated than the one of Dirichlet.
We close by an overview of open problems.
\end{abstract}

\section{Introduction}
Exact formulas and estimates of volumes of convex bodies play a central role in functional analysis and geometry of Banach spaces
with numerous connections to probability theory, approximation theory and other areas of mathematics.
We refer to the classical book \cite{Pisier} for an extensive treatment.

Unfortunately, for most of the bodies it is out of reach to express their volume by simple formulas.
The most important exception from this rule is the volume of the unit ball of the Lebesgue sequence space
$$
B_p^m(\R)=\{x=(x_1,\dots,x_m)\in\R^m:\sum_{j=1}^m |x_j|^p\le 1\}
$$
for $0<p<\infty$ (with $B_\infty^m(\R)$ being just the cube $[-1,1]^m$ with volume $2^m$).
The volume of $B_p^m(\R)$ is very well known since the time of Dirichlet, cf. \cite{Diri,Edw}. It can be expressed directly with the help of
the gamma function as
\begin{equation}\label{eq:vol}
\vol(B_p^m(\R))=2^m\frac{\Gamma\Bigl(\frac1p+1\Bigr)^m}{\Gamma\Bigl(\frac mp+1\Bigr)},
\end{equation}
where $1/p$ is interpreted as zero for $p=\infty.$

The approach of Dirichlet is based on inductive formula for evaluation of multivariate integrals. Alternatively, a noninductive way
to calculate $\vol(B_p^m(\R))$ is based on Gaussian integrals. We review briefly this approach using the notation of \cite{SZ}.
For $0<p<\infty$ we define
$$
\Delta_p^m=\{x\in\R^m:x_j\ge 0\ \text{for all}\ j=1,\dots,m\ \text{and}\ \|x\|_p=1\}
$$
and the normalized cone measure
$$
\mu_p(A)=\frac{\lambda([0,1]\cdot A)}{\lambda([0,1]\cdot \Delta_p^m)},\qquad A\subset \Delta_p^m,
$$
where $\lambda$ is the usual Lebesgue measure in $\R^m$. Applying the \emph{polar decomposition identity} \cite{BCN}
$$
\frac{\int_{\R^m_+}f(x)d\lambda(x)}{\lambda([0,1]\cdot \Delta_p^m)}=m\int_0^\infty r^{m-1}\int_{\Delta_p^m}f(rx)d\mu_p(x)dr
$$
to $f(x)=e^{-x_1^p-\dots-x_m^p}$, one gets
$$
\frac{\Bigl(\int_0^\infty e^{-t^p}dt\Bigr)^m}{2^{-m}\vol(B_p^m(\R))}=m\int_0^\infty r^{m-1}e^{-r^p}dr.
$$
After the substitution $s=t^p$  and employing the definition of the gamma function, the formula \eqref{eq:vol} follows.

The results on volumes of unit balls of finite dimensional spaces can be usually translated into the language of entropy numbers
of their embeddings. Furthermore, by the results of Carl and Triebel \cite{Carl,CarlTriebel}, they have implications on the decay
of different approximation widths as well as on the decay of eigenvalues of compact operators. Finally, by the use of discretization
techniques, these results can be exploited to study these quantities also for embeddings of function spaces, see \cite{Vyb}
for an overview of this approach and many references. When applying this approach to weighted function spaces
or function spaces with dominating mixed smoothness, entropy numbers of embeddings of mixed Lebesgue spaces $\ell_q^n(\ell_p^m)$
start to play a role, cf. \cite{Vasil,V}. We refer also to \cite{David,EdNet,Lew,Tino} for other results on finite dimensional sequence spaces with mixed norm.

Since the advent of Lasso \cite{Tib} in 1996 and its success in analysis of high dimensional data,
there was a new wave of interest in geometry of finite dimensional sequence spaces. Ten years later it was
exhibited in the theory of compressed sensing \cite{CRT2,CT2,D} that the minimization of the $\ell_1$-norm can be used for recovery of sparse
data from a small amount of random linear measurements. To recover data with more complex sparsity patterns, other $\ell_1$-based norms were also studied. The most important
functional in this connection is given by
$$
\min_{x\in\R^{n\times m}}\sum_{j=1}^n\Bigl(\sum_{k=1}^m |x_{j,k}|^2\Bigr)^{1/2},
$$
which will be denoted by $\|x\|_{2,1}$ later on. Minimization of this norm is usually denoted by
\emph{group lasso} \cite{GL2,GL1} and is used to recover data with \emph{block sparsity}, i.e. double sequences $(x_{j,k})$,
for which the set
$$
\{j\in\{1,\dots,n\}: \exists\, k\in\{1,\dots,m\}\text{\ such that\ }x_{j,k}\not=0\}
$$
has cardinality much smaller than $n$, cf. \cite{GL3}. The rest of the paper is structured as follows.
In Section 2 we evaluate an auxiliary multivariate integral, in Section 3 we prove our main result, Theorem \ref{thm_mixed},
extend the result to the complex case and discuss few special cases in detail. We close in Section 4 with a list
of open problems.

\section{Auxiliary integrals}
In this section we want to calculate the value of the integrals
\begin{align*}
 A_{m}(\alpha,\beta)=\int_{\substack{a\ge 0\\ 0\le a_1+\dots+a_m\le 1}}
(1-a_1-\dots-a_m)^\beta a_1^\alpha\dots a_m^\alpha da
\end{align*}
for $m\in\N$ and $\alpha,\beta>-1$. The integration in this notation is with respect to the set 
$$
\{a=(a_1,\dots,a_m)\in\R^m: a_j\ge 0\text{\ for all\ } j=1,\dots,m\text{\ and\ } 0\le a_1+\dots+a_m\le 1\}.
$$
For that reason we use for $t>0$ the gamma function 
\begin{align*}
 \Gamma(t)=\int_0^\infty x^{t-1}e^{-x}dx
\end{align*}
with the well known properties $\Gamma(t+1)=t\Gamma(t)$, $\Gamma(1)=1$ and $\Gamma(1/2)=\sqrt{\pi}$. Furthermore, we need the beta function and the following identity  
\begin{align*}
 B(\alpha,\beta)=\int_0^1t^{\alpha-1}(1-t)^{\beta-1}dt=\frac{\Gamma(\alpha)\Gamma(\beta)}{\Gamma(\alpha+\beta)},
\end{align*}
which holds for all $\alpha,\beta>0$, see \cite{Davis}. 
Now, we are ready to prove the following identity for $A_m(\alpha,\beta)$.
\begin{thm}\label{thm_A}
 Let $m\in\N$ and $\alpha,\beta>-1$ then
 \begin{align*}
  A_{m}(\alpha,\beta)=
\frac{\Gamma(\alpha+1)^m\Gamma(\beta+1)}{\Gamma(m(\alpha+1)+\beta+1)}.
 \end{align*}
 By choosing $\alpha=\beta>-1$ we especially have
\begin{align*}
  A_{m}(\alpha,\alpha)=\frac{\Gamma(\alpha+1)^{m+1}}{\Gamma((m+1)(\alpha+1))}.
 \end{align*}
\end{thm}
\begin{proof}
First we can observe
\begin{align}\label{A_1}
 A_{1}(\alpha,\beta)=\int_{0}^1
(1-a)^\beta a^\alpha da=B(\alpha+1,\beta+1)=\frac{\Gamma(\alpha+1)\Gamma(\beta+1)}{\Gamma(\alpha+\beta+2)}.
\end{align}

The rest of the proof follows from a recursion formula. We use the substitution $(b_1,b_2,\dots,b_{m-1},b_m)=(a_1,a_2,\dots,a_{m-1},a_1+\dots+a_m)$ with Jacobian equal to one. It follows
 \begin{align*}
  A_{m}(\alpha,\beta)&=\int_{\substack{a\ge 0\\ 0\le a_1+\dots+a_m\le 1}}(1-a_1-\dots-a_m)^\beta a_1^\alpha\dots a_m^\alpha da\\
  &=\int_{\substack{b\ge 0\\ 0\le b_1+\dots+b_{m-1}\le b_m\le 1}}(1-b_m)^\beta b_1^\alpha\dots b_{m-1}^\alpha(b_m-b_1-\dots-b_{m-1})^\alpha db\\
  &=\int_0^1(1-b_m)^\beta\int_{\substack{b'\ge 0\\ 0\le b_1+\dots+b_{m-1}\le b_m}}b_1^\alpha\dots b_{m-1}^\alpha(b_m-b_1-\dots-b_{m-1})^\alpha db'db_m,\\
  \intertext{where $b'=(b_1,\dots,b_{m-1}).$ Then we substitute $c_1=b_1/b_m,\dots,c_{m-1}=b_{m-1}/b_m$ with $\left|\frac{db'}{dc}\right|=b_m^{m-1}$, which gives}  
  &=\int_0^1 (1-b_m)^\beta b_m^{m\alpha+m-1}db_m\int_{\substack{c\ge 0\\ 0\le c_1+\dots+c_{m-1}\le 1}}c_1^\alpha\dots c_{m-1}^\alpha(1-c_1-\dots-c_{m-1})^\alpha dc\\
  &=B(m(\alpha+1),\beta+1)A_{m-1}(\alpha,\alpha).
  \end{align*}
So we have the identity
  \begin{align}\label{A_rekursion}
  A_m(\alpha,\beta)&=A_{m-1}(\alpha,\alpha)\frac{\Gamma(m(\alpha+1))\Gamma(\beta+1)}{\Gamma(m(\alpha+1)+\beta+1)}.
 \end{align}
The rest of the proof follows by induction from \eqref{A_1} and \eqref{A_rekursion}.
\end{proof}

\section{The volume of the unit ball in $\ell^n_q(\ell_p^m)$}

We are interested in the volume of the unit ball of the mixed sequence space $\ell_q^n(\ell_p^m)$. First of all we treat the real case.
\subsection{Unit ball in $\ell^n_q(\ell_p^m(\R))$}
If $x=(x_{j,k})_{j,k=1}^{n,m}$ are real numbers, then we define for all $0<p,q\leq\infty$
$$
\|x\|_{p,q}=\Biggl(\sum_{j=1}^n\Bigl(\sum_{k=1}^m |x_{j,k}|^p\Bigr)^{q/p}\Biggr)^{1/q}\quad\text{and}\quad B_{p,q}^{m,n}(\R)=\{x\in\R^{n\times m}:\|x\|_{p,q}\le 1\}
$$
with appropriate modifications if $p$ or $q$ is equal to infinity. \\
In the non-mixed case of $\ell_p^m(\R)$ the following result is well known \cite[(1.17)]{Pisier}.
\begin{prop}\label{prop}
Let $m\in\N$ and $0<p\leq\infty$ then
\begin{align}\label{eq:prop}
\vol(B_p^m(\R))=2^m\frac{\Gamma\left(\frac1p+1\right)^m}{\Gamma\left(\frac mp+1\right)}.
\end{align}
\end{prop}
For the mixed spaces $\ell_q^n(\ell_p^m(\R))$ the following theorem is our main result.
\begin{thm}\label{thm_mixed}
Let $m,n\in\N$ and $0<p,q\leq\infty$ then
\begin{align}\label{eq:mixed}
\vol(B_{p,q}^{m,n}(\R))=2^{mn}\frac{\Gamma\left(\frac1p+1\right)^{mn}}{\Gamma\left(\frac{nm}q+1\right)}\frac{\Gamma\left(\frac mq+1\right)^n}{\Gamma\left(\frac mp+1\right)^n}.
\end{align}
\end{thm}
\begin{proof}
We give the proof for $\max(p,q)<\infty.$ The modifications necessary in case of $p=\infty$ or $q=\infty$ are described in the subsection on special cases below.\\ 
First of all we restrict the volume to positive $x\geq 0$ (meaning $x_{j,k}\geq0$ for all $j,k$)
\begin{align}
\vol(B_{p,q}^{m,n}(\R))&=\int_{B_{p,q}^{m,n}(\R)}1dx=2^{mn}\int_{B_{p,q}^{m,n}(\R), x\geq0}1dx\notag
\intertext{and using the substitution $s_{j,k}=x_{j,k}^p$ we obtain}
&=\frac{2^{mn}}{p^{mn}}\int_{s\in A_{p,q}^{m,n}(1)}\prod_{j=1}^n\prod_{k=1}^m s_{j,k}^{1/p-1}ds,\label{s_int}
\end{align} 
where the integral is over $A_{p,q}^{m,n}(t)={\displaystyle\Bigl\{s\in\R^{n\times m}:s\ge 0 \text{ and }\sum_{j=1}^n\Bigl(\sum_{k=1}^ms_{j,k}\Bigr)^{q/p}\le t\Bigr\}}$. Now, we can rewrite
\begin{align}
\vol(B_{p,q}^{m,n}(\R))
&=\frac{2^{mn}}{p^{mn}}\!\!\!\int\limits_{\substack{s\geq0\\0\le s_{n,1}+\dots+s_{n,m}\le 1}}\!\!\!\!\!\!(s_{n,1}\dots s_{n,m})^{1/p-1}\!\!\!\!\!\!\int\limits_{s'\in A_{p,q}^{m,n-1}(1-(s_{n,1}+\dots+s_{n,m})^{q/p})}\prod_{j=1}^{n-1}\prod_{k=1}^m s_{j,k}^{1/p-1}ds'\,ds_{n},\notag
\intertext{where $s'=(s_{j,k}: j=1,\dots,n-1\ \text{and}\ k=1,\dots, m)$ and $s_n=(s_{n,1},\dots,s_{n,m})$. We substitute for $j<n$ with $\sigma_{j,k}=\frac{s_{j,k}}{[1-(s_{n,1}+\dots+s_{n,m})^{q/p}]^{p/q}}$ and get}
&\hspace{-0em}=\frac{2^{mn}}{p^{mn}}\int\limits_{\substack{s\geq0\\0\le s_{n,1}+\dots+s_{n,m}\le 1}}(s_{n,1}\dots s_{n,m})^{1/p-1}\left[1-(s_{n,1}+\dots+s_{n,m})^{q/p}\right]^{\frac{m(n-1)}{q}}ds_n\times\notag\\
&\qquad\qquad\qquad\times\int_{\sigma\in A_{p,q}^{m,n-1}(1)}\prod_{j=1}^{n-1}\prod_{k=1}^m \sigma_{j,k}^{1/p-1}d\sigma\notag\\
&\hspace{-0em}=\vol(B_{p,q}^{m,n-1}(\R))\frac{2^{m}}{p^{m}}\int\limits_{\substack{s\geq0\\0\le s_{1}+\dots+s_{m}\le 1}}(s_{1}\dots s_{m})^{1/p-1}\left[1-(s_{1}+\dots+s_{m})^{q/p}\right]^{\frac{m(n-1)}{q}}ds,\label{B_rekursion}
\end{align}
where the last equality follows from comparison to \eqref{s_int}. Now, we evaluate the last integral in \eqref{B_rekursion} putting $\lambda_1=s_1,\dotsc,\lambda_{m-1}=s_{m-1}$ and $\lambda_m=s_1+\dotsb+s_m$
\begin{align}
\int_{\substack{s\geq0\\0\le s_{1}+\dots+s_{m}\le 1}}&(s_{1}\dots s_{m})^{1/p-1}\left[1-(s_{1}+\dots+s_{m})^{q/p}\right]^{\frac{m(n-1)}{q}}ds\notag\\
&\hspace{-5em}=\int_{\substack{\lambda\geq0\\0\leq\lambda_1+\dotsb+\lambda_{m-1}\leq\lambda_m\leq1}}(\lambda_1\dots\lambda_{m-1})^{1/p-1}(\lambda_m-(\lambda_1+\dotsb\lambda_{m-1}))^{1/p-1}(1-\lambda_m^{q/p})^{\frac{m(n-1)}{q}}d\lambda\notag\\
&\hspace{-5em}=\int_0^1(1-\lambda_m^{q/p})^{\frac{m(n-1)}{q}}\int_{\substack{\lambda'\geq0\\0\leq\lambda_1+\dotsb+\lambda_{m-1}\leq\lambda_m}}(\lambda_1\dots\lambda_{m-1})^{1/p-1}(\lambda_m-(\lambda_1+\dotsb\lambda_{m-1}))^{1/p-1}d\lambda'\,d\lambda_m\notag
\intertext{and substituting $\mu_j=\frac{\lambda_j}{\lambda_m}$ for $j\leq m-1$ we obtain}
&\hspace{-5em}=\int_0^1(1-\lambda_m^{q/p})^{\frac{m(n-1)}{q}}\lambda_m^{m/p-1}d\lambda_m\int_{\substack{\mu\geq0\\0\leq\mu_1+\dotsb+\mu_{m-1}\leq 1}}(\mu_1\dots\mu_{m-1})^{1/p-1}(1-(\mu_1+\dotsb\mu_{m-1}))^{1/p-1}d\mu\notag\\
&\hspace{-5em}=A_{m-1}(1/p-1,1/p-1)\int_0^1(1-\lambda_m^{q/p})^{\frac{m(n-1)}{q}}\lambda_m^{m/p-1}d\lambda_m.\label{int_1}
\end{align}
Finally, we use $t=\lambda_m^{q/p}$ and see that the last integral equals
\begin{align}
\int_0^1(1-\lambda_m^{q/p})^{\frac{m(n-1)}{q}}&\lambda_m^{m/p-1}d\lambda_m
=\frac pq\int_0^1(1-t)^{\frac{m(n-1)}q}t^{\frac mq-1}dt\notag\\
&=\frac pq B\left(\frac mq,\frac{m(n-1)}q+1\right)=\frac{p}{q}\frac{\Gamma(\frac{m}{q})\Gamma(\frac{m(n-1)}{q}+1)}{\Gamma(\frac{mn}{q}+1)}\label{int_2}
\end{align}
Using the recursive formula \eqref{B_rekursion} with \eqref{int_1} and \eqref{int_2} and Theorem \ref{thm_A} we obtain
\begin{align}
\vol(B_{p,q}^{m,n}(\R))&=\vol(B_{p,q}^{m,n-1}(\R))\frac{2^{m}}{p^{m}}A_{m-1}\left(\frac1p-1,\frac1p-1\right)
\frac{p}{q}\frac{\Gamma(\frac{m}{q})\Gamma(\frac{m(n-1)}{q}+1)}{\Gamma(\frac{mn}{q}+1)}\notag\\
&=\vol(B_{p,q}^{m,n-1}(\R))\frac{2^{m}}{p^{m}}\frac pq\frac{\Gamma(\frac 1p)^m}{\Gamma(\frac mp)}\frac{\Gamma(\frac mq)\Gamma(\frac{m(n-1)}q+1)}{\Gamma(\frac{mn}q+1)}\notag\\
&=\vol(B_{p,q}^{m,n-1}(\R))2^m\frac{mp}{mq}\frac{\Gamma(\frac mq)}{\Gamma(\frac mp)}\Gamma\left(\frac 1p+1\right)^m\frac{\Gamma(\frac{m(n-1)}q+1)}{\Gamma(\frac{mn}q+1)}\notag
\intertext{and using \eqref{B_rekursion} $n-2$ times we have}
&=\vol(B_{p,q}^{m,1}(\R))\left[2^m\frac{\Gamma(\frac mq+1)}{\Gamma(\frac mp+1)}\Gamma\left(\frac 1p+1\right)^m\right]^{n-1}\frac{\Gamma(\frac{m}q+1)}{\Gamma(\frac{mn}q+1)}.\label{eq_1}
\end{align} 
We observe that $B_{p,q}^{m,1}(\R)=B_p^m(\R)$ and using Proposition \ref{prop} we finally get
\begin{align*}
\vol(B_{p,q}^{m,n}(\R))&=2^{m(n-1)}2^m\frac{\Gamma(\frac1p+1)^m}{\Gamma(\frac mp+1)}\frac{\Gamma(\frac mq+1)^{n-1}}{\Gamma(\frac mp+1)^{n-1}}\Gamma\left(\frac 1p+1\right)^{m(n-1)}\frac{\Gamma(\frac{m}q+1)}{\Gamma(\frac{mn}q+1)}\\
&=2^{mn}\frac{\Gamma(\frac mq+1)^{n}}{\Gamma(\frac mp+1)^{n}}\frac{\Gamma(\frac1p+1)^{mn}}{\Gamma(\frac{mn}q+1)},
\end{align*}
which finishes the proof.
\end{proof}

\subsection{Alternative proof}\label{sec:FB}

We present an alternative proof of \eqref{eq:mixed}, which was communicated
to us by Prof. Franck Barthe shortly after the first version of this paper
was finished.

Let $K\subset \R^d$ be a symmetric convex body with non-empty interior
and let $\|\cdot\|_K$ be the corresponding norm (i.e. its Minkowski functional).

Modifying the classical computation from convex geometry (cf. \cite[page 11]{Pisier}) we get
for smooth non-negative function $f$ on $[0,\infty)$ with fast decay at infinity
\begin{align}
\notag \int_{\R^d}f(\|x\|_K)dx&=-\int_{\R^d}\int_{\|x\|_K}^\infty f'(t)dtdx=
-\int_0^\infty \int_{x:\|x\|_K\le t}1\,dx f'(t)dt\\
\label{eq:FB4} &=-\int_0^\infty \vol(tK)f'(t)dt=-\vol(K)\int_0^\infty t^df'(t)dt\\
\notag &=\vol(K)\cdot\int_0^\infty dt^{d-1}f(t)dt.
\end{align}
Plugging $f(t)=e^{-t^p}$ into \eqref{eq:FB4}, we obtain in particular
\begin{align}
\notag\int_{\R^d}e^{-\|x\|_K^p}dx&=\vol(K)\cdot\int_0^\infty dt^{d-1}e^{-t^p}dt=\vol(K)\cdot \frac{d}{p}\cdot \int_0^\infty s^{d/p-1}e^{-s}ds\\
\label{eq:FB1}&=\frac{d\vol(K)\Gamma(d/p)}{p}=\vol(K)\Gamma(1+d/p).
\end{align}
This formula can already be used to give the volume of the $\ell^d_p$-unit ball. Indeed, we get for $K=B_p^d$
\begin{align}
\notag\vol(B_p^d)\Gamma(1+d/p)&=\int_{\R^d}e^{-\|x\|_p^p}dx=\Bigl(2\int_0^\infty e^{-t^p}dt\Bigr)^d=2^d\Bigl(\frac{1}{p}\int_0^\infty s^{1/p-1}e^{-s}ds\Bigr)^d\\
\label{eq:FB2}&=2^d\Bigl(\frac{\Gamma(1/p)}{p}\Bigr)^d=2^d\Gamma(1+1/p)^d,
\end{align}
giving \eqref{eq:vol} again. Choosing $K$ to be the unit ball of $\ell_q^n(\ell_p^m(\R))$, we get by \eqref{eq:FB1}
\begin{align}
\notag\Gamma\Bigl(1+\frac{mn}{q}\Bigr)\vol(B_{p,q}^{m,n}(\R))&=\int_{\R^{n\times m}}e^{-\|x\|^q_{p,q}}dx\\
\notag&=\int_{\R^{n\times m}}\exp\Bigl(-\sum_{j=1}^n\bigl(\sum_{k=1}^m|x_{j,k}|^p\bigr)^{q/p}\Bigr)dx\\
\label{eq:FB3}&=\prod_{j=1}^n\, \int_{\R^{m}}\exp\Bigl(-\bigl(\sum_{k=1}^m|x_{j,k}|^p\bigr)^{q/p}\Bigr)dx_{j,\cdot}\\
\notag&=\Bigl(\int_{\R^{m}}\exp\Bigl(-\bigl(\sum_{k=1}^m|x_{k}|^p\bigr)^{q/p}\Bigr)dx\Bigr)^n\\
\notag&=\Bigl(\int_{\R^{m}}e^{-\|x\|_p^{q}}dx\Bigr)^n=\Bigl[\Gamma\bigl(1+\frac{m}{q}\bigr)\vol(B_p^m)\Bigr]^n,
\end{align}
where we have used Fubini's theorem and \eqref{eq:FB1} again in the last line. Plugging \eqref{eq:FB2}
with $m$ instead of $d$ gives finally
\begin{align*}
\vol(B_{p,q}^{m,n}(\R))&=\frac{1}{\Gamma\Bigl(1+\frac{mn}{q}\Bigr)}\Bigl[\Gamma\bigl(1+\frac{m}{q}\bigr)\frac{2^m\Gamma(1+1/p)^m}{\Gamma(1+m/p)}\Bigr]^n\\
&=2^{mn}\frac{\Gamma\left(\frac1p+1\right)^{mn}}{\Gamma\left(\frac{nm}q+1\right)}\frac{\Gamma\left(\frac mq+1\right)^n}{\Gamma\left(\frac mp+1\right)^n}.
\end{align*}

As pointed to us by Aicke Hinrichs, this approach can be easily generalized to calculate
the volume of the unit ball of the space $\ell_q^n(X_j)$, where $X_j$ are $d_j$-dimensional (quasi-)Banach
spaces and $\|x\|_{\ell_q^n(X_j)}=\Bigl(\sum_{j=1}^n\|x_j\|^q_{X_j}\Bigr)^{1/q}$ giving
\begin{align}\label{eq:ah}
\vol(B_{\ell_q^n(X_j)})=\frac{\prod_{j=1}^n\Gamma(1+\frac{d_j}{q})\vol(B_{X_j})}{\Gamma\Bigl(1+\frac{\sum_{j=1}^n d_j}{q}\Bigr)}.
\end{align}

Finally, we observe that \eqref{eq:FB1} and \eqref{eq:FB3} can be easily adapted to calculate the volume of an anisotropic unit ball
\begin{align*}
 B_{p_1,\dotsc,p_n}=\{x\in\Rn:\|x\|_{p_1,\dotsc,p_n}=|x_1|^{p_1}+\dotsb+|x_n|^{p_n}\leq1\}.
\end{align*}
The gauge expression $\|x\|_{p_1,\dotsc,p_n}$ clearly satisfies
\begin{align*}
 \|tx\|_{p_1,\dotsc,p_n}=t^{p_1+\dots+p_n}\|x\|_{p_1,\dotsc,p_n}\quad\text{for }t>0.
\end{align*}
What we further need is
\begin{align}%
\notag\vol(\{x\in\Rn:\|x\|_{p_1,\dotsc,p_n}\leq t\})&=\vol(\{x\in\R^n:|x_1|^{p_1}+\dots+|x_n|^{p_n}\le t\})\\
\notag&=\vol\Bigl(\Bigl\{x\in\R^n:\Bigl|\frac{x_1}{t^{1/p_1}}\Bigr|^{p_1}+\dots+\Bigl|\frac{x_n}{t^{1/p_n}}\Bigr|^{p_n}\le 1\Bigr\}\Bigr)\\
\label{question}&=\vol(D_{\tau}B_{p_1,\dots,p_n})\\
\notag&=t^{1/p_1+\dots+1/p_n}\vol(B_{p_1,\dots,p_n})
\end{align}
for every $t>0$. Here, $D_{\tau}$ denotes a diagonal matrix with $t^{1/p_1},\dots,t^{1/p_n}$ on the diagonal.
With the help of $\eqref{question}$, we get the analogue of \eqref{eq:FB4}
\begin{align}
\notag
\int_{\R^d}f(\|x\|_{p_1,\dotsc,p_n})dx&=-\int_{\R^d}\int_{\|x\|_{p_1,\dotsc,p_n}}^\infty
f'(t)dtdx=
-\int_0^\infty \int_{x:\|x\|_{p_1,\dotsc,p_n}\le t}1\,dx f'(t)dt\\
\label{eq:mixed_n} &
=-\vol(B_{p_1,\dotsc,p_n})\int_0^\infty
t^{1/p_1+\dots1/p_n}f'(t)dt\\
\notag
&=\vol(B_{p_1,\dotsc,p_n})\left(\frac1{p_1}+\dots+\frac1{p_n}\right)\cdot\int_0^\infty
t^{1/p_1+\dots1/p_n-1}f(t)dt,
\end{align}
where $f$ is a smooth non-negative function on $[0,\infty)$ with fast decay at
infinity.\\
Now, using $f(t)=e^{-t}$ gives us
\begin{align}
 \label{eq:volmixed}
 \vol(B_{p_1,\dotsc,p_n})=2^n\frac{\Gamma(1+1/p_1)\cdot\ldots\cdot\Gamma(1+1/p_n)}{\Gamma(1+1/p_1+\dots+1/p_n)},
\end{align}
which coincides with the result obtained by Dirichlet in \cite{Diri}. Although one could easily combine for example \eqref{eq:volmixed}
with \eqref{eq:ah} to give volumes of more and more complicated bodies, we do not go into details and leave this to the interested reader.

\subsection{Unit ball in $\ell^n_q(\ell_p^m(\C))$}
If $z=(z_{j,k})_{j,k=1}^{n,m}$ are complex numbers, then we define for all $0<p,q\leq\infty$
$$
\|z\|_{p,q}=\Biggl(\sum_{j=1}^n\Bigl(\sum_{k=1}^m |z_{j,k}|^p\Bigr)^{q/p}\Biggr)^{1/q}\quad\text{and}\quad B_{p,q}^{m,n}(\C)=\{z\in\C^{n\times m}:\|z\|_{p,q}\le 1\}.
$$
Since $\C$ maybe identified with $\R^2$ we can reformulate
\begin{align*}
 B_{p,q}^{m,n}(\C)=\left\{u,v\in\R^{n\times m}:\sum_{j=1}^n\Bigl(\sum_{k=1}^m (u^2_{j,k}+v_{j,k}^2)^{p/2}\Bigr)^{q/p}\le 1\right\}.
\end{align*}
In the non-mixed case the following result holds.
\begin{prop}[Proposition 3.2.1 in \cite{EdTrie}]
Let $0<p\leq\infty$ then the volume of the unit ball in $\ell_p^m(\C)$ equals
\begin{align*}
\vol(B_p^m(\C))=\pi^m\frac{\Gamma\left(\frac2p+1\right)^m}{\Gamma\left(\frac{2m}p+1\right)}=\left(\frac{\pi}{2}\right)^m\vol\left(B_{\frac p2}^m(\R)\right).
\end{align*}
\end{prop}
A similar result holds in the mixed case.
\begin{thm}
 Let $m,n\in\N$ and $0<p,q\leq\infty$ then
 \begin{align*}
  \vol(B_{p,q}^{m,n}(\C))&=\left(\frac{\pi}{2}\right)^{mn}\vol\left(B_{\frac p2,\frac q2}^{m,n}(\R)\right)\\
	&=\pi^{mn}\frac{\Gamma\left(\frac2p+1\right)^{mn}}{\Gamma\left(\frac{2nm}q+1\right)}\frac{\Gamma\left(\frac {2m}q+1\right)^n}{\Gamma\left(\frac {2m}p+1\right)^n}.
 \end{align*}
 \end{thm}
\begin{proof}
 We start by using the substitution $u_{j,k}=r_{j,k}\cos\varphi_{j,k}$ and $v_{j,k}=r_{j,k}\sin\varphi_{j,k}$, then 
 \begin{align*}
  \vol(B_{p,q}^{m,n}(\C))&=\int_{B_{p,q}^{m,n}(\C)}1d(u,v)=(2\pi)^{mn}\int_{\displaystyle\Bigl\{r\ge 0:\sum_{j=1}^n\Bigl(\sum_{k=1}^mr_{j,k}^p\Bigr)^{q/p}\le 1\Bigr\}}\prod_{j=1}^n\prod_{k=1}^m r_{j,k}dr
  \intertext{and using $s_{j,k}=r_{j,k}^p$ we obtain with $A_{p,q}^{m,n}(t)={\displaystyle\Bigl\{s\in\R^{n\times m}:s\ge 0 \text{ and }\sum_{j=1}^n\Bigl(\sum_{k=1}^ms_{j,k}\Bigr)^{q/p}\le t\Bigr\}}$}
  &=\pi^{mn}\left(\frac 2p\right)^{mn}\int_{s\in A_{\frac p2,\frac q2}^{m,n}(1)}\prod_{j=1}^n\prod_{k=1}^m s_{j,k}^{\frac 2p-1}ds.
 \end{align*}
 Now comparing the last integral to \eqref{s_int} immediately implies
 \begin{align*}
  \vol(B_{p,q}^{m,n}(\C))=\left(\frac\pi2\right)^{mn}\vol\left(B_{\frac p2,\frac q2}^{m,n}(\R)\right).
 \end{align*}  
 \end{proof}

\subsection{Special cases}
In this section we present our main result Theorem \ref{thm_mixed} in special cases and compare it with known results. First of all we give a table combining the cases $p,q\in\{1,2,\infty\}$.\\[.5cm]
\begin{center}
\begin{tabular}{|c||c|c|c|}
\hline
\mbox{$\vol(B_{p,q}^{m,n}(\R))$} & $p=1$	&	$p=2$	&	$p=\infty$\\
\hline\hline
$q=1$	& $\frac{2^{mn}}{(mn)!}$	& $\frac{\pi^{\frac{mn}2}}{(mn)!}\frac{(m!)^n}{\Gamma(m/2+1)^n}$	& $2^{mn}\frac{(m!)^n}{(mn)!}$\\
\hline
$q=2$ & $\frac{2^{mn}}{(m!)^n}\frac{\Gamma(m/2+1)^n}{\Gamma(mn/2+1)}$	& $\frac{\pi^{\frac{mn}2}}{\Gamma(mn/2+1)}$	&$2^{mn}\frac{\Gamma(m/2+1)^n}{\Gamma(mn/2+1)}$	\\
\hline
$q=\infty$&	$\frac{2^{mn}}{(m!)^n}$&	$\frac{\pi^{\frac{mn}2}}{\Gamma(m/2+1)^n}$&$2^{mn}$\\
\hline
\end{tabular} 
\end{center}
\paragraph{\underline{$q=\infty$:}}
Now, we consider the special case of $q=\infty$, i.e.
\begin{align*}
B_{p,\infty}^{m,n}(\R)&=\left\{x\in\R^{n\times m}:\max_{1\leq j\leq n}\sum_{k=1}^m|x_{j,k}|^p\leq1\right\}
\intertext{which can be reformulated by}
&=\left\{x\in\R^{n\times m}:\sum_{k=1}^m|x_{j,k}|^p\leq1\text{ for all }1\leq j\leq n\right\}.
\end{align*}
By using Proposition \ref{prop} we can easily calculate
\begin{align*}
\vol(B_{p,\infty}^{m,n}(\R))&=\int_{x\in B_{p,\infty}^{m,n}(\R)}1dx=\left(\int_{\{x\in\R^m:\|x\|_p\leq1\}}1dx\right)^n\\
&=\left(\vol(B_p^m(\R))\right)^n=2^{mn}\frac{\Gamma(1/p+1)^{mn}}{\Gamma(m/p+1)^n},
\end{align*}
which corresponds to \eqref{eq:mixed} for $q=\infty$.

\paragraph{\underline{$p=\infty$:}}

Considering the case $p=\infty$ we can write the unit ball
\begin{align*}
B_{\infty,q}^{m,n}(\R)&=\left\{x\in\R^{n\times m}:\sum_{j=1}^n\max_{1\leq k\leq m}|x_{j,k}|^q\leq1\right\}
\end{align*}
and we get
\begin{align*}
\vol(B_{\infty,q}^{m,n}(\R))&=2^{mn}\int_{x\ge 0,x\in B_{\infty,q}^{m,n}(\R)}1dx
=2^{mn}\int_{\substack{\lambda\geq0\\0\le\lambda_1^q+\dots+\lambda_n^q\le 1}}\int_{\substack{x\ge 0\\ \max_{k=1,\dots,m}x_{j,k}=\lambda_j}}1dx\,d\lambda\\
&=2^{mn} \int_{\substack{\lambda\geq0\\0\le\lambda_1^q+\dots+\lambda_n^q\le 1}} \prod_{j=1}^n (m\lambda_j^{m-1})  d\lambda=
2^{mn} m^n\int_{\substack{\lambda\geq0\\0\le\lambda_1^q+\dots+\lambda_n^q\le 1}} \prod_{j=1}^n \lambda_j^{m-1}  d\lambda,
\end{align*}
where the last but one identity follows from the fact that the surface of the set $\{y\in\R^m: y\ge 0\ \text{and}\ \max_{k=1,\dots,m}y_k=\lambda\}$
is equal to $m\lambda^{m-1}. $

We put $\tau_j=\lambda_j^q$ 
and get
\begin{align*}
\vol(B_{\infty,q}^{m,n})&=2^{mn} m^n\int_{\substack{\tau\geq0\\0\le\tau_1+\dots+\tau_n\le 1}} \prod_{j=1}^n \Bigl(\tau_j^{(m-1)/q}\cdot\frac{1}{q\tau_j^{1-1/q}}\Bigr)  d\tau\\
&=2^{mn} \frac{m^n}{q^n}\int_{\tau\in A^{1,n}_{\frac pm,\frac qm}(1)} \prod_{j=1}^n \tau_j^{m/q-1}  d\tau.
\end{align*}
Comparing the last integral to \eqref{s_int} we see the connection to $\vol(B_{q/m}^n(\R))$ and we obtain
\begin{align*}
\vol(B_{\infty,q}^{m,n})=2^{n(m-1)}\vol(B_{q/m}^n(\R))=2^{mn}\frac{\Gamma(m/q+1)^n}{\Gamma(nm/q+1)},
\end{align*}
i.e. \eqref{eq:mixed} for $p=\infty$.
\paragraph{\underline{$p=q$:}} In this case we can write
\begin{align*}
B_{p,p}^{m,n}(\R)&=\left\{x\in\R^{n\times m}:\sum_{j=1}^n\sum_{k=1}^m|x_{j,k}|^p\leq1\right\}\\
&=B_p^{mn}(\R)
\intertext{and we easily obtain from Proposition \ref{prop} the expected result}
\vol(B_{p,p}^{m,n}(\R))&=2^{mn}\frac{\Gamma(1/p+1)^{mn}}{\Gamma(mn/p+1)}.
\end{align*}
\paragraph{\underline{$m=1$ or $n=1$:}}
In these cases we directly see
\begin{align*}
B^{1,n}_{p,q}(\R)=B^n_q(\R)\qquad\text{as well as}\qquad B^{m,1}_{p,q}(\R)=B^m_p(\R)
\end{align*}
and Proposition \ref{prop} then coincides with \eqref{eq:mixed} in Theorem \ref{thm_mixed}.

\section{Closing remarks}

The elements of the space $\R^{n\times m}$ are usually identified with $n\times m$ matrices. The most important matrix norms (like the nuclear norm or the spectral norm) are,
however, not given by simple conditions on the entries of the matrix but rather on its singular values. We refer to \cite{SR} for results on volumes of unit balls
with respect to these norms.

Theorem \ref{thm_mixed} gives the volume of the unit ball of $\ell_q^n(\ell_p^m).$
On the other hand, it leaves a number of questions open.
Let us briefly discuss the most important ones.
\begin{itemize}
\item On many occurrences in mathematics it is more convenient to work with Lorentz sequence spaces instead of Lebesgue sequence spaces.
If $0<p\le \infty$, the weak-$\ell_p^m$ (quasi-)norm is defined as
$$
\|x\|_{p,\infty}=\max_{j=1,\dots,m} j^{1/p}x_j^*,
$$
where $(x^*_j)^m_{j=1}$ is the non-increasing rearrangement of $x\in\R^m$. To our best knowledge, no closed formula for the volume of the unit ball
$\{x\in\R^m:\|x\|_{p,\infty}\le 1\}$ is available in the literature. The same holds true for the general Lorentz sequence spaces $\ell_{p,q}^m$ with $0<q<\infty$ if $q\not=p.$
\item Comparing \eqref{eq:mixed} with \eqref{eq:prop}, we observe that
\begin{equation}\label{eq:geom}
\vol(B_{p,q}^{m,n}(\R))
=\vol(B_{q}^{n\cdot m}(\R))\cdot\frac{[\vol(B_{p}^{m}(\R))]^n}{[\vol(B_{q}^{m}(\R))]^n}.
\end{equation}
This formula appeared implicitly already in \eqref{eq:FB3}. We leave it open if this formula has some geometric or combinatoric interpretation, which would allow for an non-analytical proof of Theorem \ref{thm_mixed}.
\end{itemize}

{\bf Acknowledgment:} We would like to thank Franck Barthe for careful reading of a previous version of this manuscript,
for communicating the alternative proof given in Section \ref{sec:FB} to us and for pointing us to the reference \cite{SR}.
Furthermore, we thank Aicke Hinrichs for useful discussions.


\begin{thebibliography}{99}
\bibitem{GL2} F. R. Bach, \emph{Consistency of the group lasso and multiple kernel learning}, J. Mach. Learn. Res. \textbf{9} (2008), 1179--1225.
\bibitem{BCN} F. Barthe, M. Cs\"ornyei, A. Naor, \emph{A note on simultaneous polar and Cartesian decomposition}, in: Geometric Aspects of Functional Analysis,
Lecture Notes in Mathematics, Springer, Berlin, 2003.
\bibitem{Carl} B. Carl, \emph{Entropy numbers, s-numbers, and eigenvalue problems}, J. Funct. Anal. 41 (1981), no. 3, 290--306.
\bibitem{CarlTriebel}  B. Carl, H. Triebel, \emph{Inequalities between eigenvalues, entropy numbers, and related quantities of compact operators in Banach spaces}.
Math. Ann. \textbf{251} (1980), no. 2, 129--133.
\bibitem{CRT2} E. J. Cand\'es, J. Romberg, T. Tao, \emph{Stable signal recovery from incomplete and inaccurate measurements}, Comm. Pure Appl. Math. \textbf{59} (2006), 1207--1223.
\bibitem{CT2} E. J. Cand\'es, T. Tao, \emph{Near-optimal signal recovery from random projections: universal encoding strategies?}, IEEE Trans. Inform. Theory \textbf{52} (2006), 5406--5425.
\bibitem{David} I. A. Davidovich, \emph {Widths of function classes and sets of finite volume in spaces with mixed norms}, Dissertation, Moscow, 1995.
\bibitem{Davis} P.J. Davis, \emph{6. Gamma function and related functions}, in M. Abramowitz, I.A. Stegun, Handbook of Mathematical Functions with Formulas, Graphs, and Mathematical Tables, New York: Dover Publications, 1972.
\bibitem{Diri} J.P.G.L. Dirichlet, \emph{Sur une nouvelle m\'ethode pour la d\'etermination des int\'egrales multiples},
Journal de Math\'ematiques Pures et Appliqu\'ees \textbf{4} (1839), 164--168.
\bibitem{D} D. L. Donoho, \emph{Compressed sensing}, IEEE Trans. Inform. Theory \textbf{52} (2006), 1289--1306.
\bibitem{EdTrie} D. E. Edmunds, H. Triebel, \emph{Function spaces, entropy numbers, differential operators}, Cambridge Tracts in Math. \textbf{120}, 1996.
\bibitem{EdNet} D. E. Edmunds, Yu. Netrusov, \emph{Sch\"utt’s theorem for vector-valued sequence spaces}, J. Appr. Theory \textbf{178} (2014), 13--21.
\bibitem{Edw} J. Edwards, \emph{A Treatise on the Integral Calculus}, Vol. II, Chelsea Publishing Company, New York, 1922.
\bibitem{Lew} L. Markhasin, \emph{Discrepancy and integration in function spaces with dominating mixed smoothness}, Friedrich-Schiller-Universit\"at Jena, Dissertation, 2012.
\bibitem{Tino} S. Mayer, S. Dirksen, T. Ullrich, \emph{Gelfand and entropy numbers of $\ell_p(\ell_q)$ balls}, in preparation.
\bibitem{GL1} L. Meier, S. van de Geer, P. B\"uhlmann, \emph{The group Lasso for logistic regression}, J. R. Stat. Soc. B \textbf{70} (2008), no. 1, 53--71.
\bibitem{Pisier} G. Pisier, \emph{The volume of convex bodies and Banach space geometry}, Cambridge Tracts in Math. \textbf{94}, 1989. 
\bibitem{SR} J. Saint-Raymond, \emph{Le volume des id\'eaux d'op\'erateurs classiques}, Studia Math.  \textbf{80}  (1984), no. 1, 63--75.
\bibitem{SZ} G. Schechtman and J. Zinn, \emph{On the volume of the intersection of two $L^n_p$ balls}, Proc. AMS \textbf{110} (1990), (1), 217--224.
\bibitem{Tib} R. Tibshirani, \emph{Regression shrinkage and selection via the Lasso}, J. R. Stat. Soc. B \textbf{58} (1996), 267--288. 
\bibitem{Vasil} A.A. Vasil'eva, \emph{Kolmogorov and linear widths of the weighted Besov classes with singularity at the origin}, J. Appr. Theory \textbf{167} (2013), 1--41.
\bibitem{V} J.~Vyb\'\i ral, \emph{Function spaces with dominating mixed smoothness}, Diss. Math. \textbf{436} (2006), 1--73.
\bibitem{Vyb} J.~Vyb\'\i ral, \emph{Widths of embeddings in function spaces}, J. Compl. \textbf{24} (2008), 545--570.
\bibitem{GL3} M. Yuan, Y. Lin, \emph{Model selection and estimation in regression with grouped variables}, J. R. Stat. Soc. B \textbf{68} (2006), (1), 49--67.
\end{thebibliography}
\end{document}